\documentclass[12pt]{article}%
\usepackage{amsfonts}
\usepackage{eurosym}
\usepackage{makeidx}
\usepackage{multirow}
\usepackage{multicol}
\usepackage[dvipsnames,svgnames,table]{xcolor}
\usepackage{graphicx}
\usepackage{epstopdf}
\usepackage{ulem}
\usepackage{hyperref}
\usepackage{amsmath}
\usepackage{amssymb}
\usepackage{tikz}
\usepackage{cite} 
\definecolor {refcol}{RGB}{40,0,255}
\hypersetup{colorlinks=true,allcolors=refcol}
\makeatletter
\newfont{\footsc}{cmcsc10 at 8truept}
\newfont{\footbf}{cmbx10 at 8truept}
\newfont{\footrm}{cmr10 at 10truept}
\newtheorem{theorem}{Theorem}[section]

\numberwithin{equation}{section}

\newtheorem{conjecture}{Conjecture}
\newtheorem{corollary}{Corollary}[section]

\newtheorem{example}{Example}

\newtheorem{remark}{Remark}

\newenvironment{proof}[1][Proof]{\noindent{\textbf {#1}}}  {\hfill$\Box$\bigskip}

\begin{document}
\title{\textbf{On the Randi\'c energy of caterpillar graphs}}
\author{Domingos M. Cardoso\thanks{Center for Research and Development in Mathematics and Applications, Department of Mathematics, Universidade de  Aveiro, Aveiro, Portugal.}, Paula Carvalho\thanks{Department of Mathematics, Universidade de  Aveiro, Aveiro, Portugal.}, Roberto C. D\'{i}az\thanks{Departamento de Matematicas, Universidad de La Serena, Cisternas 1200, La Serena, Chile.}, Paula Rama\thanks{Department of Mathematics, Universidade de  Aveiro, Aveiro, Portugal.}}

\date{}
\maketitle

\begin{abstract}
A caterpillar graph $T(p_1, \ldots, p_r)$ of order $n= r+\sum_{i=1}^r p_i$, $r\geq 2$, is a tree such that removing all its pendent vertices gives rise to a path of order $r$. In this paper we establish a necessary and sufficient condition for a real number to be an eigenvalue of the Randi\'c matrix of $T(p_1, \ldots, p_r)$. This result is applied to determine the extremal  caterpillars for the Randi\'c energy of $T(p_1,\ldots, p_r)$  for cases $r=2$ (the double star) and $r=3$. We characterize the extremal  caterpillars  for $r=2$. Moreover, we study the family of caterpillars $T\big(p,n-p-q-3,q\big)$ of order $n$, where $q$ is a function of $p$, and we characterize the extremal  caterpillars  for three cases: $q=p$, $q=n-p-b-3$ and $q=b$, for $b\in \{1,\ldots,n-6\}$ fixed. Some illustrative examples are included.
\end{abstract}

\textbf{AMS classification:} \textit{05C50, 15A18}

\textbf{Keywords:} \textit{Randi\'{c} matrix, Caterpillars graphs, Energy of graphs, Randi\'c energy.}

\section{Introduction}
It is worth to start this section defining the Randi\'{c} matrix of a graph $G$, denoted by $R_{G} =\left( r_{ij}\right)$, which is such that $r_{ij} = \frac{1}{\sqrt{d_id_j}}$ if $ij \in E(G)$ and zero otherwise, where $d_k$ is the degree of the vertex $k$. The spectrum of $R_G$ is the multiset of its eigenvalues, $\sigma_R(G) =\{\rho_1^{[m_1]} , \rho_2^{[m_2]}, \ldots, \rho_s^{[m_s]}\} $, where $m_i$ stands for the multiplicity of $\rho_i$, for $1 \le i \le s$, and $\rho_1 > \rho_2 > \cdots > \rho_s$ are the distinct eigenvalues of $R_{G}$.\\

It is well known that $\rho_1(G)=1$ whenever $G$ is a graph with at least one edge (see \cite[Th. 2.3]{GutmanFurtulaBozkurt14}).

The \textit{Randi\'c energy} of a graph $G$ is defined in \cite{GutmanFurtulaBozkurt14} (see also
\cite{BozkurtGungorGutman10-A, BozkurtGungorGutman10-B}) as follows:
\begin{equation*} \label{renergy}
RE(G) = \sum_{i=1}^{n}{|\rho_i(G)|}.
\end{equation*}
It is immediate that $RE(G)=0$ if and only if all the vertices of $G$ are isolated vertices.\\
Considering $\lambda_1 \ge \lambda_2 \ge \cdots \ge \lambda_n$ as the eigenvalues of the adjacency matrix of a graph $G$ of order $n$, the ordinary energy of $G$ \cite{Gutman78, LiShiGutman2012}, herein denoted by $\mathcal{E}(G)$, is defined as
$$
\mathcal{E}(G) = \sum_{j=1}^{n}{|\lambda_i|}.
$$
In \cite{GutmanFurtulaBozkurt14}, the Randi\'c energy and the ordinary energy of the paths $P_n$ and $P_{n-2}$, respectively, are related as follows.
$$
RE(P_n) = 2 + \frac{1}{2}\mathcal{E}(P_{n-2}).
$$

According to \cite{CaversFallatKirkland10}, if a graph $G$ of order $n$ has at least one edge, then
\begin{equation}
2 \le RE(G) \le n. \label{lower_and_upper_bound}
\end{equation}

Furthermore, the lower bound in \eqref{lower_and_upper_bound} is attained if and only if one component of $G$ is a complete multipartite graph and all other components (if any) are isolated vertices. In particular, $RE(G)=2$ for complete graphs. The upper bound in \eqref{lower_and_upper_bound} is attained only if $n$ is even and $G$ is isomorphic to $\frac{n}{2}K_2$, or $n$ is odd and $G$ is the disjoin union of $\frac{n-3}{2}$ $K_2$ plus a component which is a path $P_2$ or a triangle $K_3$.

The characterization of connected graphs with maximal Randi\'c energy remains an open problem as well as the following conjecture posed in \cite{GutmanFurtulaBozkurt14} and  computationally verified for graphs of order $n$ up to $n=10$.
\begin{conjecture}\cite{GutmanFurtulaBozkurt14}\label{MaximalEnergyConjecture1}
	The connected graph with maximal Randi\'c energy is a tree.
\end{conjecture}

The following more thinner conjecture, also posed in \cite{GutmanFurtulaBozkurt14}, remains open too.

\begin{conjecture}\cite{GutmanFurtulaBozkurt14}\label{MaximalEnergyConjecture2}
	The connected graph of odd order $n \ge 1$, having maximal Randi\'c energy is the sun \cite[Fig. 2]{GutmanFurtulaBozkurt14}.
	The connected graph of even order $n \ge 2$, having maximal Randi\'c energy is the balanced double sun \cite[Fig. 2]{GutmanFurtulaBozkurt14}.
\end{conjecture}

The aim of this paper is to determine the  extremal graphs for the Randi\'c energy of a family of caterpillars $T(p_1, \cdots p_r)$  of order $n= r+\sum_{i=1}^r p_i$  for cases $r=2$ and $r=3$. The paper is organized as follows.
In Section~\ref{notation} the notation and basic definitions of the main concepts used through the text are introduced. In Section~~\ref{randic_spectra} a caterpillar is considera as the H-join of graphs and some spectral results of graphs obtained by this operation are recalled. Moreover, we get a necessary and sufficient condition for a real number to be an eigenvalue of the Randi\'c matrix. This result plays a important role throughout  the paper. In Section~\ref{extremal_caterpillars} we characterize the extremal caterpillar graphs for $r=2$ (that are the double star) as well as we study the family of caterpillars  $T\big(p,n-p-q-3,q\big)$ of order $n$, and we characterize extremal caterpillar graphs for three cases: $q=p$, $q=n-p-b-3$ and   $q=b$, for any $b\in \{1,\ldots,n-6\}$ fixed.

\section{Preliminaries}\label{notation}
In this paper we deal with undirected simple graphs. For a graph $G$ the vertex set is denoted by $V(G)$ and the edge set by $E(G)$ and $|V(G)|$ is the order of $G$. The edges of $G$ denoted by $ij$, where $i$ and $j$ are the end-vertices of the edge. When $ij \in E(G)$ we say that the vertices $i$ and $j$ are adjacent and also that $i$ is a neighbor of $j$ (and conversely).  The neighborhood of a vertex $v  \in V(G)$ is the set of its neighbors and is denoted by $N_{G}(v) = \{w: vw \in E(G)\}$. The degree of $v$, denoted by $d_v$, is the cardinality of $N_{G}(v)$. The vertices $i$ with $0$ degree are called isolated vertices. Two graphs $G$ and $H$ are isomorphic if there is a bijection $\psi: V(G) \rightarrow V(H)$ such that $ij \in E(G)$ if and only if $\psi(i)\psi(j) \in E(H)$. This binary relation between graphs is denoted by $G \cong H$. The complement graph of a graph $G$, denoted by $\overline{G}$, is such that $V(\overline{G}) = V(G)$ and $E(\overline{G})=\{ij: ij \not \in E(G)\}$. The complete graph of order $n$, denote by $K_n$, is a graph where every pair of vertices are adjacent. The vertices of the complement of $K_n$ are all isolated. The adjacency matrix of a graph $G$ of order $n =|V(G)|$ is $n \times n$ symmetric matrix $A_{G}=(a_{ij})$ such that $a_{ij}=1$ if   $ij \in E(G)$ and zero  otherwise.
The spectrum of a matrix $M$ is the multiset of its eigenvalues denoted by  $\sigma_M$. In particular, the spectrum of the adjacency matrix of a graph $G$, also called the spectrum of $G$, is
$\sigma(G) =\{\lambda_1^{[m_1]} , \lambda_2^{[m_2]}, \ldots, \lambda_s^{[m_s]}\}$, where $m_i$ stands for the multiplicity of $\lambda_i$, for $1 \le i \le s$.

A path with $r$ vertices, denoted by $P_r$, is a sequence of vertices $v_1,v_2, \ldots, v_r$ such that  each vertex is adjacent to the next, that is  $v_1v_{i+1} \in E(G)$ for $i=1, \ldots, r-1$. A cycle $C_{r}$ is a closed path with $r$ edges, that is, such that $v_{r+1}=v_1$.
A tree is a connected acyclic graph; a star of order $r$, denoted by $S_{r+1}$,  is a tree with a central vertex with degree $r$ and all the other $r$ vertices are pendent. A caterpillar is a tree such that removing all pendent vertices give rise to a path with at least two vertices. In particular, $T(p_1,\ldots, p_r)$ denotes a caterpillar obtained by attaching  the central vertex of a star $S_{p_i+1}$ to the $i$-th vertex of  $P_r$, $i=1, \ldots r$. The order of a caterpillar is $n= r+ \sum_{i=1}^r p_i$.

A  caterpillar $T(p_1,\ldots, p_r)$ can also be seen as the H-join
$H[G_1, \dots, G_r,$ $G_{r+1}, \dots, G_{2r}]$, where, for
$1 \le i \le r, \; \left\{\begin{array}{lcl}
                                G_i     & \cong & K_1 \\
                                G_{i+r} & \cong & \overline{K_{p_i}}
                         \end{array}\right.$ and $H$ is the caterpillar of order $2r$, $T(1, \dots, 1)$,  that is, a path $P_r$ with one pendant vertex attached to each vertex of the path.

The null square and the identity matrices of order $n$ are denoted by $O_{n}$ and $I_n$, respectively.

\section{The Randi\'c spectrum of a caterpillar viewed as $H$-join} \label{randic_spectra}

In this section, we consider a caterpillar as the H-join of a family of graphs  (see \cite{CardosoFreitasMartinsRobbiano}),  $T(p_1,\ldots, p_r)$ =
$H[K_1, \ldots, K_1, \overline{K_{p_1}}, \ldots, \overline{K_{p_r}}]$, where $H$ is the caterpillar of order $2r$, $T(1,1, \dots,1)$, that is, a path $P_r$ with a pendant edge attached to each vertex of the path.
The following result, given in \cite{Andrade_et_al17}, characterizes Randi\'c spectra of $H$-join graphs.

\begin{theorem}\cite{Andrade_et_al17} \label{Andrade} 
	Let $H$ be a graph of order $k$. Let $G_j$ be a $d_j$-regular graph of order $n_j$, with $d_j \ge 0$, $n_j \ge 1$, for $j=1, \ldots, k$ and
	$G=H[G_1, \ldots, G_k]$. Let $R_G$ be the Randi\'c 	matrix of $G$. Then,
	\begin{equation*}
	\sigma_{R_G} =
	\sigma_{\Gamma_k}  \cup \bigcup_{j=1}^k \left\{\frac{\lambda}{N_j+d_j}: \lambda \in \sigma\left(A_{G_{j}}\right)\setminus\{d_j\}\right\},
\end{equation*}
	where $N_j = \sum\limits_{i \in N_{H}(j)}{n_i}$, for $j=1, 2, \dots, k$,
	\begin{equation*}
	\Gamma_k = \left(\begin{array}{ccccc}
	\frac{d_1}{N_1+d_1} & \rho_{12}           & \dots & \rho_{1(k-1)} & \rho_{1k} \\
	\rho_{12}           & \frac{d_2}{N_2+d_2} & \dots & \rho_{2(k-1)} & \rho_{2k} \\
	\vdots              & \vdots              & \ddots & \vdots        & \vdots \\
	\rho_{1(k-1)}       & \rho_{2(k-1)}       & \dots & \frac{d_{k-1}}{N_{k-1}+d_{k-1}} & \rho_{(k-1)k} \\
	\rho_{1k}           & \rho_{2k}           & \dots & \rho_{(k-1)k} & \frac{d_{k}}{N_{k}+d_{k}}\\
	\end{array}
	\right)
	\end{equation*}
	and
\begin{equation*}
\rho_{ij} = \delta_{ij}\frac{\sqrt{n_i n_j}}{\sqrt{(N_i+d_i)(N_j+d_j)}},
 \end{equation*}
 with $\delta_{ij}=1$  if $ij \in E(H)$, and zero otherwise, for $ i= 1 \ldots, k-1$ and $j= i+1, \ldots, k$.
\end{theorem}

\begin{remark}\label{rem1}
{\rm It is clear that the Randi\'c matrix of a $d_j$-regular graph $G_j$ is
$R_{G_j}= \frac{1}{d_j}A_{G_j}$ if $d_j>0$ and zero otherwise. On the other hand, if $d_j=0,$ for $j=1, \dots, k$, then $\Gamma_k = \Omega A_H \Omega$, with
$\Omega = \text{diag}\left\{\sqrt{\frac{n_1}{N_1}}, \dots, \sqrt{\frac{n_k}{N_k}}\right\}$.}
\end{remark}

Since $K_1$ and $\overline{K_{p_i}}$, for $i=1,\ldots,r$, are $0$-regular graphs,
we have the following result, which  plays an important role in this paper:
\begin{corollary}\label{Cor1}
Let $H =T(1,1, \ldots, 1)$ be the caterpillar of order $2r$, $r \ge 2$, obtained from a path $P_r$ and a pendent vertex attached to each vertex of the path.
Let  $T = T(p_1, \ldots, p_r) = H[K_1, \dots, K_1, \overline{K_{p_1}}, \dots, \overline{K_{p_r}}]$
be a caterpillar of order $ n= r+\sum\limits_{i=1}^r p_i$. Then,
\begin{equation*}
\sigma_{R_T}= \sigma_{\Gamma_{2r}}
\cup \left\{  0^{[ \sum_{i=1}^r (p_i-1)]}\right\}.
\end{equation*}
\end{corollary}

As a consequence, in order to obtain the spectrum of the Randi\'c matrix of $ T(p_1, \ldots, p_r) $ we focus our attention on the spectrum of $\Gamma_{2r}$.
Firstly, note that
$$
\Omega = \text{diag}\left\{\sqrt{\frac{1}{N_1}}, \dots, \sqrt{\frac{1}{N_r}},\sqrt{\frac{p_1}{N_{r+1}}}, \dots , \sqrt{\frac{p_r}{N_{2r}}}\right\} = \begin{bmatrix}
 \Omega_{1}    & O_{r} \\
 O_{r} & \Omega_{2}
 \end{bmatrix}
$$
with
\begin{equation} \label{Gamma2r0}
\Omega_{1} = \text{diag}\left\{\sqrt{\frac{1}{N_1}}, \dots, \sqrt{\frac{1}{N_{r}}}\right\}, \qquad
\Omega_{2} = \text{diag}\left\{\sqrt{\frac{p_1}{N_{r+1}}}, \dots, \sqrt{\frac{p_r}{N_{2r}}}\right\},
\end{equation}

Therefore, we can write
\begin{equation} \label{Gamma2r1}
\Gamma_{2r}
 = \Omega A_H \Omega
 = \begin{bmatrix}
 \Omega_{1}    & O_{r} \\
 O_{r} & \Omega_{2}
 \end{bmatrix}
 \begin{bmatrix} A_{P_r} & I_r \\
   I_r & O_{r}
 \end{bmatrix}
 \begin{bmatrix}
 \Omega_{1}    & O_{r} \\
 O_{r} & \Omega_{2}
 \end{bmatrix}
 = \begin{bmatrix} A & B \\
        B & O_{r}
     \end{bmatrix},
\end{equation}
where
\begin{equation}\label{AB}
A=\Omega_{1} A_{P_{r}} \Omega_{1} \quad
\text{and} \quad
B=\Omega_{1} \Omega_{2}.
\end{equation}\

It is worth to recall a famous determinantal identity presented by Issa Schur in 1917 \cite{schur1917}
referred as the formula of Schur by Gantmacher \cite[p. 46]{Gantmacher1998}. In the sixties, the term Schur complement was introduced by Emilie Haynsworth \cite{Haynsworth1968} jointly with the following notation. Considering a square matrix
$M=\begin{bmatrix} A & B \\  C & D \end{bmatrix}$, where $A$ and $D$ are square block matrices and $A$ is nonsingular, the Schur complement of $A$ in $M$ is defined as
\begin{equation*}
M/A = D - CA^{-1}B.\label{schur_complement}
\end{equation*}
For more details see \cite{HornZhang2005}. Using the above notation, the next theorem states the Schur determinantal identity. For the readers convenience, the very short proof presented in \cite{HornZhang2005} is reproduced.

\begin{theorem}\cite{schur1917}\label{Imp}
Let \ $M=\begin{bmatrix} A & B \\  C & D \end{bmatrix}$, where $A$ and $D$ are square submatrices of order $m$ and $n$,
respectively. If $A$ is nonsingular then
\begin{align*}
\det(M)=det(A) \cdot \det(M/A).
\end{align*}
\end{theorem}

\begin{proof} \
It is immediate that
$$
\begin{bmatrix} A & B \\  C & D \end{bmatrix} = \begin{bmatrix} I_m & 0 \\  CA^{-1} & I_n        \end{bmatrix}
                                                \begin{bmatrix}  A  & B \\  0       & D-CA^{-1}B \end{bmatrix}.
$$
The identity follows by taking the determinant of both sides.
\end{proof}

Similarly, if $D$ is nonsingular then
\begin{align}
\det(M)=\det(A - BD^{-1}C) \cdot \det(D). \label{schur_formula}
\end{align}
Note that
$
\begin{bmatrix} A & B \\  C & D \end{bmatrix} = \begin{bmatrix} I_m & BD^{-1}    \\  0 & I_n \end{bmatrix}
                                          \begin{bmatrix} A- BD^{-1}C  & 0 \\  C & D   \end{bmatrix}.
$

From \eqref{schur_formula}, we may establish the following spectral characterization for the matrix $\Gamma_{2r}$, which will play an important role in getting our main results:
\begin{theorem}\label{novo}
Let $H =T(1,1, \ldots, 1)$ be the caterpillar of order $2r$, $r \ge 2$ and let $\Gamma_{2r}$ be partitioned as in \eqref{Gamma2r1}. Then, $\lambda\in \sigma_{\Gamma_{2r}}$ if and only if
\begin{align*}
det(\lambda^{2} I_{r}-\lambda A-B^{2})=0,
\end{align*}
where $A$ and $B$ are defined as in \eqref{AB}.
\end{theorem}
\begin{proof} \
The characteristic polynomial of  $\Gamma_{2r}$ is
\begin{eqnarray*}
p_{\Gamma_{2r}}(\lambda)
&=& \operatorname{det}(\lambda I_{2r}-\Gamma_{2r})
=det\left(
\begin{bmatrix}
	\lambda I_{r}-A & -B \\
	-B & \lambda I_{r}
\end{bmatrix}\right).
\end{eqnarray*}
Thus, applying \eqref{schur_formula}, we obtain
\begin{eqnarray*}
p_{\Gamma_{2r}}(\lambda)
&=&det(\lambda I_{r})\cdot det\left(\lambda I_{r}-A-B\left(\frac{1}{\lambda} I_{r}\right)B\right) \\ \\
&=&\lambda^{r}\cdot det\left(\left(\frac{1}{\lambda}\right)(\lambda^{2} I_{r}-\lambda A-B^{2})\right) \\ \\
&=&\lambda^{r}\cdot \left(\frac{1}{\lambda}\right)^{r}\cdot det\left(\lambda^{2} I_{r}-\lambda A-B^{2}\right) = det\left(\lambda^{2} I_{r}-\lambda A-B^{2}\right).
\end{eqnarray*}
\end{proof}

\section{Extremal caterpillar graphs for Randi\'c energy}\label{extremal_caterpillars}
In this section, we obtain the extremal graphs in the family of caterpillars, for $r=2,3$.
\subsection{
	Extremal caterpillar graphs $\mathbf{T(p,n-p-2)}$, $\mathbf{p=1, \ldots, \lfloor \frac{n-2}{2} \rfloor}$.}
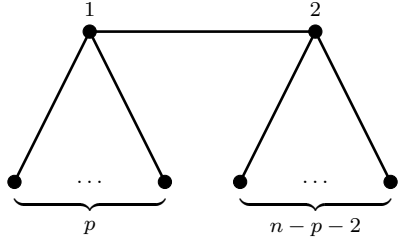
\begin{center}
\begin{tikzpicture}
\draw [line width=1.pt] (0.,0.)-- (-1.,-2.);
\draw [line width=1.pt] (0.,0.)-- (1.,-2.);
\draw [line width=1.pt] (0.,0.)-- (3.,0.);
\draw [line width=1.pt] (3.,0.)-- (2.,-2.);
\draw [line width=1.pt] (3.,0.)-- (4.,-2.);
\begin{scriptsize}
\draw [fill=black] (0.,0.) circle (2.5pt);
\draw (0.,0.3) node {$1$};
\draw [fill=black] (-1.,-2.) circle (2.5pt);
\draw [fill=black] (1.,-2.) circle (2.5pt);
\draw [fill=black] (3.,0.) circle (2.5pt);
\draw (3.,0.3) node {$2$};
\draw [fill=black] (2.,-2.) circle (2.5pt);
\draw [fill=black] (4.,-2.) circle (2.5pt);
\draw (0.,-2.) node {$\ldots$};
\draw(0.,-2.3) node {$\underbrace{\hspace{2cm}}$};
\draw (0.,-2.6) node {$p$};
\draw (3.,-2.) node {$\ldots$};
\draw(3.,-2.3) node {$\underbrace{\hspace{2cm}}$};
\draw (3.,-2.6) node {$n-p-2$};
\end{scriptsize}
\end{tikzpicture}
\end{center}

\begin{theorem}
Let $T_{p}=T(p,n-p-2)$, $p=1, \ldots, \lfloor \frac{n-2}{2} \rfloor$ be a caterpillar of order $n\geq 4$. 	
Then
\begin{equation*}
2+\sqrt{ \frac{2(n-3)}{n-2}} \ \leq \ RE(T_{p}) \ \leq \ 4-\frac{4}{n}.
\end{equation*}
The lower bound is attained if and only if $p=1$ (the graph obtained by attaching a pendent vertex to a pendent vertex of $S_{n-1}$) and the upper bound is attained if and only if $T_{p}$ has even order and $p=\frac{n-2}{2}$.
\end{theorem}
\begin{proof} \
By Theorem \ref{novo}, the eigenvalues of $\sigma_{\Gamma_{4}}$ are the zeros of the polynomial
$ \operatorname{det}(\lambda^{2} I_{2}-\lambda A-B^{2})=0$
where (see (\ref{AB})),
\[
A=\begin{bmatrix}
 0 & \frac{1}{\sqrt{(p+1)(n-p-1)}} \\  \frac{1}{\sqrt{(p+1)(n-p-1)}} & 0
 \end{bmatrix}
\hbox{ and }
B =\begin{bmatrix}
\frac{\sqrt{p}}{\sqrt{p+1}} & 0 \\
0 & \frac{\sqrt{n-p-2}}{\sqrt{n-p-1}}
\end{bmatrix}.
\]
So,
\begin{eqnarray*}
\operatorname{det}(\lambda^{2} I_{2}-\lambda A-B^{2})
&=&\operatorname{det}\begin{bmatrix}
	\lambda^{2}-\frac{p}{p+1} & -\frac{\lambda}{\sqrt{(p+1)(n-p-1)}} \\ -\frac{\lambda}{\sqrt{(p+1)(n-p-1)}} & \lambda^{2}-\frac{n-p-2}{n-p-1}
   \end{bmatrix} \\
&=&\left(\lambda^{2}-\frac{p}{p+1}\right)\left(\lambda^{2}-\frac{n-p-2}{n-p-1}\right)-\frac{\lambda^{2}}{(p+1)(n-p-1)}
\\
&=&\frac{\lambda^{2}-1}{(p+1)(n-p-1)} \left((p+1)(n-p-1)\lambda^{2}-p(n-p-2)\right).
\end{eqnarray*}
Consequently,
\begin{equation*}
\sigma_{\Gamma_{4}}=\left\{\pm \sqrt{\frac{p(n-p-2)}{(p+1)(n-p-1)}},  \pm 1\right\}.
\end{equation*}
and
\begin{equation*}\label{Imp2}
RE(T_{p})
=\sum_{i=1}^{n}{\vert \lambda_{i}(T_{p}) \vert}
=2+2\sqrt{\frac{p(n-p-2)}{(p+1)(n-p-1)}},
\end{equation*}
for all $p=1,\ldots,\left\lfloor\frac{n-2}{2}\right\rfloor$.
For
$1\le x \le \left\lfloor\frac{n-2}{2}\right\rfloor$,
let $f(x)=\frac{x\,(n-x-2)}{(x+1)(n-x-1)}$.  Then,
\[
f^{\prime}(x)=\frac{\big(n-1\big)\big(n-2(x+1)\big)}{(x+1)^2(n-x-1)^2} \ge 0 .
\]
if and only if
$1\le x \le \frac{n-2}{2}$. Therefore, $f$ is an increasing function in this interval,  and consequently,
\begin{align*}
2+\sqrt{\frac{2(n-3)}{n-2}} \leq \ RE(T_{p}) \leq  RE(T_{\lfloor\frac{n-2}{2}\rfloor}),
\end{align*}
for all \ $p=1,\ldots,\left\lfloor\frac{n-2}{2}\right\rfloor$. Finally, if $n$ is even,
\[
RE(T_{\lfloor\frac{n-2}{2}\rfloor})
=RE(T_{\frac{n-2}{2}})
=2+2\left(\frac{n-2}{n}\right)
= 4 -\frac{4}{n},
\]
and if $n$ is odd, 
\begin{align*}
RE(T_{\lfloor\frac{n-2}{2}\rfloor})
=RE(T_{\lfloor\frac{n-3}{2}+\frac{1}{2}\rfloor})
=RE(T_{\frac{n-3}{2}})
=2+2\left(\sqrt{\frac{n-3}{n+2}}\right)   < 4 -\frac{4}{n}
\end{align*}
for all $n \geq 3$.
\end{proof}

\subsection{Extremal caterpillar graphs $\mathbf{T\big(p,n-p-q-3,q\big)}$, $\mathbf{p, q\in \{1,\ldots,n-5\}}$}

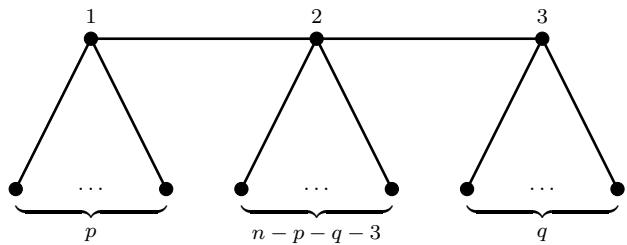
\begin{center}
\begin{tikzpicture}
\draw [line width=1.pt] (0.,0.)-- (-1.,-2.);
\draw [line width=1.pt] (0.,0.)-- (1.,-2.);
\draw [line width=1.pt] (0.,0.)-- (3.,0.);
\draw [line width=1.pt] (3.,0.)-- (2.,-2.);
\draw [line width=1.pt] (3.,0.)-- (4.,-2.);
\draw [line width=1.pt] (3.,0.)-- (6.,0.);
\draw [line width=1.pt] (6.,0.)-- (5.,-2.);
\draw [line width=1.pt] (6.,0.)-- (7.,-2.);
\begin{scriptsize}
\draw [fill=black] (0.,0.) circle (2.5pt);
\draw (0.,0.3) node {$1$};
\draw [fill=black] (-1.,-2.) circle (2.5pt);
\draw [fill=black] (1.,-2.) circle (2.5pt);
\draw [fill=black] (3.,0.) circle (2.5pt);
\draw (3.,0.3) node {$2$};
\draw [fill=black] (2.,-2.) circle (2.5pt);
\draw [fill=black] (4.,-2.) circle (2.5pt);
\draw [fill=black] (6.,0.) circle (2.5pt);
\draw (6.,0.3) node {$3$};
\draw [fill=black] (5.,-2.) circle (2.5pt);
\draw [fill=black] (7.,-2.) circle (2.5pt);
\draw (0.,-2.) node {$\ldots$};
\draw(0.,-2.3) node {$\underbrace{\hspace{2cm}}$};
\draw (0.,-2.6) node {$p$};
\draw (3.,-2.) node {$\ldots$};
\draw(3.,-2.3) node {$\underbrace{\hspace{2cm}}$};
\draw (3.,-2.6) node {$n-p-q-3$};
\draw (6.,-2.) node {$\ldots$};
\draw(6.,-2.3) node {$\underbrace{\hspace{2cm}}$};
\draw (6.,-2.6) node {$q$};
\end{scriptsize}
\end{tikzpicture}
\end{center}
For this class of caterpillars,
\[
\Omega_{1}
= \operatorname{diag} \left\{\frac{1}{\sqrt{p+1}}, \frac{1}{\sqrt{n-p-q-1}}, \frac{1}{\sqrt{q+1}}\right\}\]
 and
\[\Omega_{2}
=\operatorname{diag}\left\{\sqrt{p}, \sqrt{n-p-q-3}, \sqrt{q}\right\}.\]
Therefore (see \eqref{Gamma2r0},  \eqref{Gamma2r1} and \eqref{AB})
\[
\Gamma_{6} = \Omega A_{H} \Omega = \begin{bmatrix}
A & B \\
B & O_{3}
\end{bmatrix},
\]
with
 \begin{equation*} \label{A3}
A = \Omega_{1}A_{P_{3}}\Omega_{1} = \begin{bmatrix}
 0 & \frac{1}{\sqrt{p+1}\sqrt{n-p-q-1}} & 0 \\
 \frac{1}{\sqrt{p+1}\sqrt{n-p-q-1}} & 0 & \frac{1}{\sqrt{q+1}\sqrt{n-p-q-1}} \\
0 & \frac{1}{\sqrt{q+1}\sqrt{n-p-q-1}} & 0
\end{bmatrix}
\end{equation*}
and
 \begin{equation*} \label{B3}
B= \Omega_{1}\Omega_{2} = \begin{bmatrix}
\frac{\sqrt{p}}{\sqrt{p+1}} & 0 & 0 \\
0 & \frac{\sqrt{n-p-q-3}}{\sqrt{n-p-q-1}} & 0 \\
0 & 0 & \frac{\sqrt{q}}{\sqrt{q+1}}
\end{bmatrix}.
\end{equation*}
By Theorem \ref{novo}, as

\[\lambda^{2} I_{3}-\lambda A-B^{2} =  \begin{bmatrix} \displaystyle
 \lambda^{2}-\frac{p}{p+1} & -\frac{\lambda}{\sqrt{p+1}\sqrt{n-p-q-1}} & 0 \\
 -\frac{\lambda}{\sqrt{p+1}\sqrt{n-p-q-1}} & \lambda^{2}-\frac{n-p-q-3}{n-p-q-1} & -\frac{\lambda}{\sqrt{q+1}\sqrt{n-p-q-1}} \\
0 &-\frac{\lambda}{\sqrt{q+1}\sqrt{n-p-q-1}} & \lambda^{2}-\frac{q}{q+1}
\end{bmatrix},\]
{\footnotesize
\begin{align*}
&\operatorname{det}(\lambda^{2} I_{3}-\lambda A-B^{2}) = \\ \\
&=\left(\lambda^{2}-\frac{p}{p+1}\right)\operatorname{det}\left(\begin{bmatrix} \lambda^{2}-\frac{n-p-q-3}{n-p-q-1} & -\frac{\lambda}{\sqrt{q+1}\sqrt{n-p-q-1}} \\ -\frac{\lambda}{\sqrt{q+1}\sqrt{n-p-q-1}} & \lambda^{2}-\frac{q}{q+1}\end{bmatrix}\right) \\ \\ &+\left(\frac{\lambda}{\sqrt{p+1}\sqrt{n-p-q-1}}\right)\operatorname{det}\left(\begin{bmatrix} -\frac{\lambda}{\sqrt{p+1}\sqrt{n-p-q-1}} & -\frac{\lambda}{\sqrt{q+1}\sqrt{n-p-q-1}} \\ 0 & \lambda^{2}-\frac{q}{q+1}\end{bmatrix}\right) \\ \\
&=\left(\lambda^{2}-\frac{p}{p+1}\right)\left[\left(\lambda^{2}-\frac{n-p-q-3}{n-p-q-1}\right)\left(\lambda^{2}-\frac{q}{q+1}\right)-
\frac{\lambda^{2}}{(q+1)(n-p-q-1)}\right] \\ \\
&-\left(\frac{\lambda^{2}}{(p+1)(n-p-q-1)}\right)\left(\lambda^{2}-\frac{q}{q+1}\right) \\ \\
&=\frac{\left(\lambda^{2}(p+1)-p\right)\left[\left(\lambda^{2}(n-p-q-1)-(n-p-q-3)\right)\left(\lambda^{2}(q+1)-q\right)-\lambda^{2}\right]-\lambda^{2}\left(\lambda^{2}(q+1)-q\right)}{(p+1)(q+1)(n-p-q-1)}. \end{align*}
}\ \\
After some algebraic manipulation on the above expression, we get that \\
{\footnotesize
\begin{align*}
\operatorname{det}(\lambda^{2} I_{3}-\lambda A-B^{2}) &= \frac{1}{(p+1)(q+1)(n-p-q-1)}\bigg[(p+1)(q+1)(n-p-q-1)\lambda^{6}
\\ \\
&-\left[(n-p-q-2)\left(q(2p+1)+p\right)+(p+1)(q+1)(n-p-q-1)\right]\lambda^{4}
\\ \\
&+\left[pq(n-p-q-3)+(n-p-q-2)\left(q(2p+1)+p\right)\right]\lambda^{2}-pq(n-p-q-3)\bigg]
\end{align*}
\begin{align*}
&=\frac{(\lambda^{2}-1)\bigg[(p+1)(q+1)(n-p-q-1)\lambda^{4}-(n-p-q-2)\left(q(2p+1)+p\right)\lambda^{2}+pq(n-p-q-3)\bigg]}{(p+1)(q+1)(n-p-q-1)}
\\ \\
&=\frac{1}{\eta(n,p,q)}\big(\lambda^{2}-1\big)\bigg[\eta(n,p,q)\lambda^{4}-\zeta(n,p,q)\lambda^{2}+\chi(n,p,q)\bigg],
\end{align*}
}
being
\begin{align}\label{Imp3}
\begin{cases}
\eta(n,p,q)&=(p+1)(q+1)(n-p-q-1), \\ \\
\zeta(n,p,q)&=(n-p-q-2)\left(q(2p+1)+p\right) \\ \\
\chi(n,p,q)&=pq(n-p-q-3).
\end{cases}
\end{align}
When obtaining the roots of the biquadratic equation
\begin{equation*}
\eta(n,p,q)\lambda^{4}-\zeta(n,p,q)\lambda^{2}+\chi(n,p,q)=0,
\end{equation*}
we determinate the roots of the equation $\operatorname{det}(\lambda^{2} I_{3}-\lambda A-B^{2}) = 0$, given by:
\begin{eqnarray*}
\lambda_{1,2} &=& \pm 1 \\ \\
\lambda_{3,4} &=& \pm \sqrt{\frac{\zeta(n,p,q)+\sqrt{\zeta^{2}(n,p,q)-4\eta(n,p,q)\chi(n,p,q)}}{2\eta(n,p,q)}}\\ \\
\lambda_{5,6} &=& \pm \sqrt{\frac{\zeta(n,p,q)-\sqrt{\zeta^{2}(n,p,q)-4\eta(n,p,q)\chi(n,p,q)}}{2\eta(n,p,q)}}.
\end{eqnarray*}\ \\
Using the notation
\begin{align}\label{RD}
\begin{cases}
\alpha(n,p,q)&=\frac{\zeta(n,p,q)}{2\eta(n,p,q)}, \\ \\
\gamma(n,p,q)&=\frac{\chi(n,p,q)}{\eta(n,p,q)}, \\ \\
\beta(n,p,q)&=\sqrt{\alpha^{2}(n,p,q)-\gamma(n,p,q)},
\end{cases}
\end{align}
we get, for a general caterpillar $T_{p,q} =T(p,n-3-p-q,q)$,
{\footnotesize
\begin{equation}
	\label{RGeneral}
RE(T_{p,q}) = 2\left(1 + \sqrt{\alpha(n,p,q)+\beta(n,p,q)} + \sqrt{\alpha(n,p,q)-\beta(n,p,q)}\right).
\end{equation}}

In order to obtain the extreme graphs for certain subfamilies of caterpillar of the form $T(p,n-p-q-3,q)$, for $n\geq 7$,
we consider $q$ as a function of $x$ such that $1 \le q(x) \le n-5$ for $1 \le x \le n-5$ and define
\begin{equation} \label{f*}
f(x) = \sqrt{\alpha(x)+\beta(x)}+\sqrt{\alpha(x)-\beta(x)},
\end{equation}
where $\alpha(x):=\alpha(n,x,q(x))$ and $\beta(x)=\sqrt{\alpha^{2}(x)-\gamma(x)}:=\beta(n,x,q(x))$ as in \eqref{RD}. Therefore,
{\footnotesize
\begin{eqnarray*}
f'(x) &=& \frac{1}{2} \left(
\frac{\alpha'(x)+\beta'(x)}{\sqrt{\alpha(x)+\beta(x)}}+
\frac{\alpha'(x)-\beta'(x)}{\sqrt{\alpha(x)-\beta(x)}}
\right)\\
&=& \frac{1}{2} \left(
\frac{f(x)\alpha'(x)+
	(\sqrt{\alpha(x)-\beta(x) }- \sqrt{\alpha(x)+\beta(x) })\beta'(x) }
{\sqrt{\gamma(x)}}
\right)\\
 &=&
\frac{1}{2} \left(
\frac{f^2(x)\alpha'(x)-2 \beta(x) \beta'(x) }
{f(x)\sqrt{\gamma(x)}}
\right)\\
&=&
\frac{\alpha'(x) \left(f^2(x) -2 \alpha(x)\right)  + \gamma'(x) }
     {2f(x)\sqrt{\gamma(x)}}\\
&=& \frac{2\alpha'(x) \sqrt{\gamma(x)}+ \gamma'(x)}
     {2f(x)\sqrt{\gamma(x)}},
\end{eqnarray*}}
where, $\gamma(x):=\gamma(n,x,q(x))$. So,
\begin{align}\label{Imp1}
f'(x)\geq 0 \ \ \ \ \text{if and only if} \ \ \ \ \lambda(x) := 2\alpha'(x) \sqrt{\gamma(x)}+ \gamma'(x) \geq 0.
\end{align}

Taking into account \eqref{Imp3} and \eqref{RD}, it is easy to see that $0\le\gamma(x)<1$, for all $1\leq x \leq n-5$.

Thus,
\begin{itemize}
\item[i.] If \ $\alpha^{\prime}(x)\geq 0$ \ and \ $\gamma^{\prime}(x)\leq 0$, \ for $x\in I\subset [1, n-5]$, \ then by \eqref{Imp1}
\begin{align}\label{Des2}
\gamma^{\prime}(x) \ \leq \ \lambda(x) \ < \ 2\alpha^{\prime}(x).
\end{align}

\item[ii.] If \ $\alpha^{\prime}(x)\leq 0$ \ and \ $\gamma^{\prime}(x)\geq 0$, \ for $x\in I\subset [1, n-5]$, \ then by \eqref{Imp1}
\begin{align}\label{Des3}
2\alpha^{\prime}(x) \ < \ \lambda(x) \ \leq \ \gamma^{\prime}(x).
\end{align}
\end{itemize}

Next we characterize the extremal caterpillars $T(p,n-2p-3,q)$ for three specific cases:
$q=p$, $q=n-p-b-3$ and   $q=b$, for any $b\in \{1,\ldots,n-6\}$ fixed.

\subsubsection{Extremal graphs for the family of caterpillars $\mathbf{T(p,b,n-p-b-3)}$}
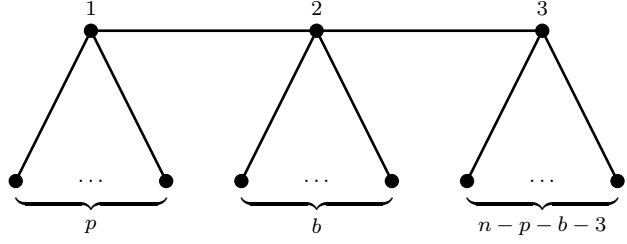
\begin{center}
	\begin{tikzpicture}
		\draw [line width=1.pt] (0.,0.)-- (-1.,-2.);
		\draw [line width=1.pt] (0.,0.)-- (1.,-2.);
		\draw [line width=1.pt] (0.,0.)-- (3.,0.);
		\draw [line width=1.pt] (3.,0.)-- (2.,-2.);
		\draw [line width=1.pt] (3.,0.)-- (4.,-2.);
		\draw [line width=1.pt] (3.,0.)-- (6.,0.);
		\draw [line width=1.pt] (6.,0.)-- (5.,-2.);
		\draw [line width=1.pt] (6.,0.)-- (7.,-2.);
		\begin{scriptsize}
			\draw [fill=black] (0.,0.) circle (2.5pt);
			\draw (0.,0.3) node {$1$};
			\draw [fill=black] (-1.,-2.) circle (2.5pt);
			\draw [fill=black] (1.,-2.) circle (2.5pt);
			\draw [fill=black] (3.,0.) circle (2.5pt);
			\draw (3.,0.3) node {$2$};
			\draw [fill=black] (2.,-2.) circle (2.5pt);
			\draw [fill=black] (4.,-2.) circle (2.5pt);
			\draw [fill=black] (6.,0.) circle (2.5pt);
			\draw (6.,0.3) node {$3$};
			\draw [fill=black] (5.,-2.) circle (2.5pt);
			\draw [fill=black] (7.,-2.) circle (2.5pt);
			\draw (0.,-2.) node {$\ldots$};
			\draw(0.,-2.3) node {$\underbrace{\hspace{2cm}}$};
			\draw (0.,-2.6) node {$p$};
			\draw (3.,-2.) node {$\ldots$};
			\draw(3.,-2.3) node {$\underbrace{\hspace{2cm}}$};
			\draw (3.,-2.6) node {$b$};
			\draw(6.,-2.) node {$\ldots$};
			\draw(6.,-2.3) node {$\underbrace{\hspace{2cm}}$};
			\draw (6.,-2.6) node {$n-p-b-3$};
		\end{scriptsize}
	\end{tikzpicture}
\end{center}

\begin{theorem} \label{caseA}
	Let $T_p =T(p,b,n-p-b-3)$ be a caterpillar of order $n \ge 7$, with $b\in \{1, \ldots, n-6\}$ fixed and $p=1,\ldots,n-b-4$.
	Then
	\[
	RE(T_{1}) \le RE(T_p) \le RE\left(T_{\lfloor\frac{n-b-3}{2}\rfloor}\right).
	\]
\end{theorem}
\begin{proof} \
	Without loss of generality, we take $1\le p\le \lfloor\frac{n-b-3}{2}\rfloor$, since for $p=1,\ldots,n-b-4$, $T_p$ and $T_{n-p-b-3}$ are isomorphic graphs. Replacing $q$ by $n-p-b-3$ in \eqref{RGeneral}, and considering the function  $f(x)$, as in  \eqref{f*},  for $1\leq x \leq \frac{n-b-3}{2}$,
	\[
	\alpha'(x) = \frac{(b+1)(n-b-1)(n-2x-b-3)}{2(b+2)\big((x+1)(n-x-b-2)\big)^{2}},\]
	
	\[\gamma'(x) = \frac{b(n-b-2)(n-2x-b-3)}{(b+2)(x+1)^2(n-x-b-2)^{2}}.\]  
	we have both $\alpha'(x)\geq 0$ and $\gamma'(x)\geq 0$ if and only if $1 \leq x \leq \frac{n-b-3}{2}$. Thus, by \eqref{Imp1}, $f$  increases in the interval  $[1, \frac{n-b-3}{2}]$ and the proof is complete.
\end{proof}

\subsubsection{Extremal graphs for the family of caterpillar $\mathbf{T(p,n-2p-3,p)}$}
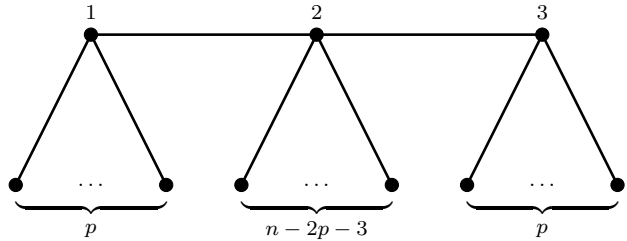
\begin{center}
\begin{tikzpicture}
\draw [line width=1.pt] (0.,0.)-- (-1.,-2.);
\draw [line width=1.pt] (0.,0.)-- (1.,-2.);
\draw [line width=1.pt] (0.,0.)-- (3.,0.);
\draw [line width=1.pt] (3.,0.)-- (2.,-2.);
\draw [line width=1.pt] (3.,0.)-- (4.,-2.);
\draw [line width=1.pt] (3.,0.)-- (6.,0.);
\draw [line width=1.pt] (6.,0.)-- (5.,-2.);
\draw [line width=1.pt] (6.,0.)-- (7.,-2.);
\begin{scriptsize}
\draw [fill=black] (0.,0.) circle (2.5pt);
\draw (0.,0.3) node {$1$};
\draw [fill=black] (-1.,-2.) circle (2.5pt);
\draw [fill=black] (1.,-2.) circle (2.5pt);
\draw [fill=black] (3.,0.) circle (2.5pt);
\draw (3.,0.3) node {$2$};
\draw [fill=black] (2.,-2.) circle (2.5pt);
\draw [fill=black] (4.,-2.) circle (2.5pt);
\draw [fill=black] (6.,0.) circle (2.5pt);
\draw (6.,0.3) node {$3$};
\draw [fill=black] (5.,-2.) circle (2.5pt);
\draw [fill=black] (7.,-2.) circle (2.5pt);
\draw (0.,-2.) node {$\ldots$};
\draw(0.,-2.3) node {$\underbrace{\hspace{2cm}}$};
\draw (0.,-2.6) node {$p$};
\draw (3.,-2.) node {$\ldots$};
\draw(3.,-2.3) node {$\underbrace{\hspace{2cm}}$};
\draw (3.,-2.6) node {$n-2p-3$};
\draw (6.,-2.) node {$\ldots$};
\draw(6.,-2.3) node {$\underbrace{\hspace{2cm}}$};
\draw (6.,-2.6) node {$p$};
\end{scriptsize}
\end{tikzpicture}
\end{center}

\begin{theorem}	\label{equi}
Let $T_p =T(p,n-2p-3,p)$ be a caterpillar of order $n \ge 7$, with $p=1,\ldots, \lfloor \frac{n-4}{2} \rfloor$.
Then
	$$RE(T_{1})\le \ RE(T_{p})\le RE(T_{z}),$$   where
	 $z$ is an integer number in $I=\left[round(r), \ round(s)\right],$
	with
	\[
	r =  \frac{1}{2}\left(2n-3-\sqrt{2n(n-2)+3}\right)  \quad \hbox{ and } \quad
	s =  \frac{1}{2}\left(2(n-1)-\sqrt{2n(n-1)}\right).\]	

\end{theorem}
\begin{proof} \
From \eqref{RGeneral}, replacing $q$ by $p$, consider (see \eqref{f*})
$f(x)$ for $1\leq x \leq \frac{n-4}{2}$.
The derivatives of $\alpha$ and $\gamma$ are
\[
\alpha'(x)
=\frac{2x^{2}-(4n-4)x+n^{2}-3n+2}{(x+1)^2(n-2x-1)^{2}}
\]
and
\[
\gamma'(x)
=\frac{2x\big(2x^{2}-(4n-6)x+n^{2}-4n+3\big)}{(x+1)^{3}(n-2x-1)^{2}}
\]
thus, we have
$\alpha'(x)\geq 0$ \ if only if \
$ 2x^{2}-(4n-4)x+n^{2}-3n+2 \geq 0$
which occurs for
$ x \le s_1$ or $x \ge s_2$ with $s_{1,2} \ = \ \frac{1}{2}\left(2(n-1)\mp \sqrt{2n(n-1)}\right).$

Similarly,
$\gamma'(x)  \ge 0$ if only if $x\big(2x^{2}-(4n-6)x+n^{2}-4n+3\big) \ge 0$, that is, for
$ 0 \le x \le r_1$ or $x \ge r_2$ with
$r_{1,2} \ = \ \frac{1}{2}\left(2n-3\mp \sqrt{2n(n-2)+3}\right).$

We have
	$
	1<r_{1}<s_{1} < \frac{n-4}{2}  < r_{2}, s_{2}.
	$
 Therefore (see \eqref{Imp1}, \eqref{Des2} and \eqref{Des3}), $f$  increases in the interval $[1, \ r_{1}]$ and decreases in $[s_{1},\frac{n-4}{2}]$.
By Bolzano's Theorem, there exists $\bar{z}\in (r_{1}, s_{1})$ such that $f'(\bar{z})=0$.
Since $s_{1}- r_{1} < 0.5$, we take $z=round(\bar{z})\in \left[round(r), \ round(s)\right]$, where $r= r_1$ and $s=s_1$.
Finally, $f(1) < f(\frac{n-4}{2})$, so $f(1)$ is the minimum of this function.
\end{proof}

\begin{example}
	A table with some values for $RE(T_{z-1})$, $RE(T_{z})$, $RE(T_{z+1})$ and  extremal caterpillars $T(p,n-3-2p, p)$ are presented below.
	
\[
	\begin{array}{cccccccc}
		n &r        &s         &z& RE(T_{z-1}) & RE(T_{z})  & RE(T_{z+1}) &  extremal\; graph\\
		\hline
		19& 4.762261& 4.923303 &5 & 5.388854 & 5.406881 & 5.363498 & T(5,6,5) \\
		21& 5.349028& 5.508623 &5 & 5.421848 & 5.458735 & 5.455208 & T(5,8,5)  \\
		35& 9.453171& 9.607378 &10& 5.672191 & 5.672395 & 5.662869 & T(10,12,10) \\
		50& 13.84816&14.000000 &14& 5.768798 & 5.770057 & 5.768229 & T(14,19,14)
	\end{array}
\]

\end{example}	

\subsubsection{Extremal graphs of the family of caterpillars $\mathbf{T(p,n-p-b-3,b)}$}
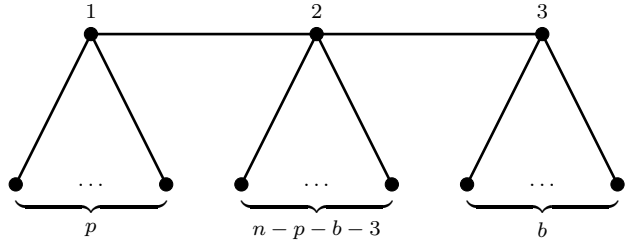
\begin{center}
\begin{tikzpicture}
\draw [line width=1.pt] (0.,0.)-- (-1.,-2.);
\draw [line width=1.pt] (0.,0.)-- (1.,-2.);
\draw [line width=1.pt] (0.,0.)-- (3.,0.);
\draw [line width=1.pt] (3.,0.)-- (2.,-2.);
\draw [line width=1.pt] (3.,0.)-- (4.,-2.);
\draw [line width=1.pt] (3.,0.)-- (6.,0.);
\draw [line width=1.pt] (6.,0.)-- (5.,-2.);
\draw [line width=1.pt] (6.,0.)-- (7.,-2.);
\begin{scriptsize}
\draw [fill=black] (0.,0.) circle (2.5pt);
\draw (0.,0.3) node {$1$};
\draw [fill=black] (-1.,-2.) circle (2.5pt);
\draw [fill=black] (1.,-2.) circle (2.5pt);
\draw [fill=black] (3.,0.) circle (2.5pt);
\draw (3.,0.3) node {$2$};
\draw [fill=black] (2.,-2.) circle (2.5pt);
\draw [fill=black] (4.,-2.) circle (2.5pt);
\draw [fill=black] (6.,0.) circle (2.5pt);
\draw (6.,0.3) node {$3$};
\draw [fill=black] (5.,-2.) circle (2.5pt);
\draw [fill=black] (7.,-2.) circle (2.5pt);
\draw (0.,-2.) node {$\ldots$};
\draw(0.,-2.3) node {$\underbrace{\hspace{2cm}}$};
\draw (0.,-2.6) node {$p$};
\draw (3.,-2.) node {$\ldots$};
\draw(3.,-2.3) node {$\underbrace{\hspace{2cm}}$};
\draw (3.,-2.6) node {$n-p-b-3$};
\draw(6.,-2.) node {$\ldots$};
\draw(6.,-2.3) node {$\underbrace{\hspace{2cm}}$};
\draw (6.,-2.6) node {$b$};
\end{scriptsize}
\end{tikzpicture}
\end{center}

\begin{theorem}\label{TRD}
Let $T_p =T(p,n-p-b-3,b)$ be a caterpillar of order $n \ge 7$, with $b\in \{1, \ldots, n-6\}$ fixed and $p=1,\ldots, n-b-4$.
Then,
\[
RE(T_{n-b-4}) \leq \ RE(T_{p}) \leq RE(T_{ z}), \hbox{ where }   z\in I=\left[round(r), \ round(s)\right],
\]
with
\begin{equation}	\label{p0}
r = -\big(n-b-1\big) + \sqrt{2(n-b-1)(n-b-2)}
\end{equation}
and
{\footnotesize
\begin{equation}  \label{p1}
s=  \frac{1}{b} \bigg(-\big((b+1)(n-b)-1\big)+\sqrt{(b+1)(n-b-1)\big((2b+1)(n-1)-2b^{2}\big)}\bigg).
\end{equation}}
\end{theorem}

\begin{proof} \
Replacing $q$ by $b$ in \eqref{RGeneral},
we define  $f(x)$ as in \eqref{f*}  for $1\leq x \leq n-b-4$.
We compute
$$
\alpha'(x) = \frac{-bx^{2}-2\big((b+1)(n-b)-1\big)x+(b+1)n^{2}-(b+1)(2b+3)n+b(b+2)^{2}+2}{2(b+1)\big((x+1)(n-x-b-1)\big)^{2}},$$

$$\gamma'(x) = \frac{b\big(-x^{2}-2(n-b-1)x+n^{2}-2(b+2)n+(b+3)(b+1)\big)}{(b+1)\big((x+1)(n-x-b-1)\big)^{2}}.
$$
We have $\alpha'(x)\geq 0$  if only if
$s_{1} \le x \le s_{2}$ with
$$
s_{1,2} =  \frac{1}{b} \bigg(-\big((b+1)(n-b)-1\big)\mp \sqrt{(b+1)(n-b-1)\big((2b+1)(n-1)-2b^{2}\big)}\bigg),
$$
and $\gamma'(x)\geq 0$ if only if
$r_{1}\le x \le r_{2}$ with
\begin{equation*}
r_{1,2} = -\big(n-b-1\big) \mp \sqrt{2(n-b-1)(n-b-2)}.
\end{equation*}
We  have $ 1 < r_{2} < s_{2} <  n-b-4.$ So, for $s=s_2$ and $r=r_2$, we get that (see \eqref{Imp1}, \eqref{Des2} and \eqref{Des3})
$f$ is increasing in   $[1, r]$ and decreasing in $[s, n-b-4]$. Therefore, there exists $\bar{z}\in (r, s)$ such that $f'(\bar{z})=0$. So, we take $z=round(\bar{z})\in I=\left[round(r), \ round(s)\right]$.
Furthermore, $f(n-b-4) < f(1)$, which complets the proof.
\end{proof}

\begin{example}
To obtain the maximal Randi\'c energy caterpillar graphs $T(p_1,p_2,p_3)$ of order $n=33$, we apply Theorems $5$, $6$ and $7$, for slight different values  of $b$, shown in the following table.

\[
\begin{array}{ccccccccc}
Theorem         & b & r      &s         &z& RE(T_{z})      &  extremal\; graph\\
\hline
\ref{caseA}     & 12&        &          &9 &5.653986727    & T( 9,12 ,9) \\
\ref{caseA}     & 9 &        &          &10&5.639354482    & T( 10,9 ,11) \\
\ref{equi}      & 9 &8.867059& 9.021749 &9 & 5.653986727   & T( 9,12 ,9) \\
\ref{TRD}       & 9 &8.811947& 9.031236 &9 & 5.653986727   & T( 9,12 ,9) \\
\ref{TRD}       & 8 &9.226495& 9.469988 &9 & 5.652375900   & T( 9,13 ,8) \\
\ref{TRD}       & 10&8.397368& 8.597041 &9 & 5.651878107   & T( 8,12 ,10) \\
\end{array}
\]

\end{example}

\begin{remark}
		{\rm In Theorem \ref{TRD}, we find a estimated interval
		$$I=\left[round(r), \ round(s)\right],$$
		where $r$ and $s$ are given in \eqref{p0} and \eqref{p1}, respectively, which contains the value of $z$ that maximizes Randi\'c energy for the family of caterpillars $T_p =T(p,n-p-b-3,b)$, $n \ge 7$, with $b\in \{1, \ldots, n-6\}$ fixed, for each $p=1,\ldots, n-b-4$. In this case, we want to point out that the interval $I$ does not necessarily has range less than $1$.  In fact, that interval have length less than $1$ if and only if
		\begin{equation} \nonumber 
		 g(n,b)=8(n+b-1)^{2}(n-b-1)(n-b-2)-(3n^{2}-3bn-9n+2b+6)^{2}>0,
		\end{equation}
and this function $g(n,b)$ can be written as:
\begin{eqnarray*}
	g(n,b)
	&=& 8(n+b-1)^2(n-b-1)(n-b-2)-\Big(3(n-1)(n-2)-(3n-2)b\Big)^2.
\end{eqnarray*}
Since $n-b-1>n-b-2>0$ then  $g(n,b)>h(n,b)$, with
\begin{eqnarray*}
	h(n,b) &=& 8(n+b-1)^2(n-b-2)^2-\Big(3(n-1)(n-2)-(3n-2)b\Big)^2\\
	&=& \Big(\sqrt{8}(n+b-1)(n-b-2)-3(n-1)(n-2)+(3n-2)b\Big)\\
	&\times & \Big(\sqrt{8}(n+b-1)(n-b-2)+3(n-1)(n-2)-(3n-2)b\Big).
\end{eqnarray*}
For each $n$, we find the values of $b$ such that
\begin{eqnarray*}
	\Delta_1 &=& \sqrt{8}(n+b-1)(n-b-2)-3(n-1)(n-2)+(3n-2)b>0, \mbox{ and, }\\
	\Delta_2 & = & \sqrt{8}(n+b-1)(n-b-2)+3(n-1)(n-2)-(3n-2)b>0.
\end{eqnarray*}
Taking into account that $3n-2>3(n-1)>0$, then
\begin{eqnarray*}
	\Delta_1 &>& \sqrt{8}(n+b-1)(n-b-2)-3(n-1)(n-2)+3(n-1)b\\
	&=& \Big(\sqrt{8}b-(3-\sqrt{8})(n-1)\Big)(n-b-2).
\end{eqnarray*}

Since
$$\sqrt{8}b-(3-\sqrt{8})(n-1)>0\Leftrightarrow b>\frac{(3-\sqrt{8})}{\sqrt{8}}(n-1)\simeq 0.06066(n-1),$$
for such values of $b$, $\Delta_1>0.$
Now, let us show that
$$\Delta_2 =\sqrt{8}(n+b-1)(n-b-2)-(3n-2)b+3(n-1)(n-2)>0.$$
From
$$-(3n-2)b+3(n-1)(n-2)>0 \Leftrightarrow  b<\frac{3(n-1)(n-2)}{3n-2}$$
and
$$ n-6<\frac{3(n-1)(n-2)}{3n-2} \Leftrightarrow 3n^3-20n+12<3(n^2-3n+2)\Leftrightarrow 6<11 n\quad \mbox{ (which is true),}$$
it follows that $\Delta_2>0$, for \ $1\le b\le n-6<\frac{3(n-1)(n-2)}{3n-2}.$
\\
From the above, for $n\ge 7$ and $b\in\mathbb{N}$ such that \;$0.06066(n-1)\le b\le n-6$,
$$
g(n,b)>h(n,b)=\Delta_1\,\Delta_2>0.
$$
Given $n\ge 7$, consider $b_{\min}$ the smallest integer $b\ge 1$ such that $g(n,b)>0$ and let $b^*=0.06066(n-1)$. For different values of $n$, $b^*$ remais close to the exact value $b_{\min}$:
$$
\begin{array}{r|r|c}
	n    & b_{\min} & b^*\\
	\hline
	20   & 1   & 1.1525\\
	30   & 2   & 1.7591\\
	50   & 3   & 2.9723\\
	100  & 6   & 6.0053\\
	500  & 30  & 30.269\\
	1000 & 61  &  60.599\\
	5000 & 303 & 303.24\\
	10000 & 606 & 606.54\\
	20000 & 1213 & 1213.1\\
\end{array}
$$
}
\end{remark}

\bigskip

\noindent \textbf{Acknowledgements:}
This research is partially supported by the Portuguese Foundation for Science
and Technology (\textquotedblleft FCT-Funda\c c\~ao para a Ci\^encia e a Tecnologia \textquotedblright),
through the CIDMA - Center for Research and Development in Mathematics and Applications, within project
UID/MAT/04106/2013. The research of R. C. D\'iaz was supported by Conicyt-Fondecyt de Postdoctorado 2017 {\rm $N^{o}$} 3170065, Chile.

\bigskip
\begin{flushleft}
\textbf{Domingos M. Cardoso} \\
\texttt{Department of Mathematics} \\
\texttt{Universidade de Aveiro} \\
\texttt{Aveiro, Portugal} \\
\texttt{E-mails}: \textit{dcardoso@ua.pt}
\end{flushleft}

\bigskip
\begin{flushleft}
\textbf{Paula Carvalho} \\
\texttt{Department of Mathematics} \\
\texttt{Universidade de Aveiro} \\
\texttt{Aveiro, Portugal} \\
\texttt{E-mails}: \textit{paula.carvalho@ua.pt}
\end{flushleft}

\bigskip
\begin{flushleft}
\textbf{Roberto C. D\'{i}az} \\
\texttt{Departamento de Matem\'aticas}, \\
\texttt{Universidad de La Serena} \\
\texttt{La Serena, Chile} \\
\texttt{E-mails}: \textit{roberto.diazm@userena.cl}
\end{flushleft}

\bigskip
\begin{flushleft}
\textbf{Paula Rama} \\
\texttt{Department of Mathematics} \\
\texttt{Universidade de Aveiro} \\
\texttt{Aveiro, Portugal} \\
\texttt{E-mails}: \textit{prama@ua.pt }
\end{flushleft}

\end{document}